\documentclass[11pt,thmsa,reqno]{amsart}

\usepackage[all]{xy}
\usepackage{color}
\usepackage{mathtools} %for long exact sequences
\usepackage{fancybox}
\usepackage{amssymb}
\usepackage{hyperref}
\usepackage{comment}
\usepackage{graphicx}
\usepackage{faktor}
\usepackage[normalem]{ulem}

 \newtheorem{theorem}{Theorem}[section]
 \newtheorem*{thm*}{Theorem} 
 \newtheorem{lemma}[theorem]{Lemma}
 \newtheorem{proposition}[theorem]{Proposition}
 
 \newtheorem{corollary}[theorem]{Corollary}
 \newtheorem{definition}[theorem]{Definition}
  \theoremstyle{definition}
 
 \newtheorem{remark}[theorem]{Remark}

\newcommand{\Presym}{\mathsf{Pre}\textrm{-}\mathsf{Sym}}

\newcommand{\hor}{\mathrm{hor}}

\renewcommand{\graph}{\mathrm{graph}}
\newcommand{\Fol}{K}
\newcommand{\Foliations}{\mathsf{Fol}}

\newcommand{\id}{\mathrm{id}}

\newcommand{\Lie}{\mathcal{L}}

\newcommand{\Diff}{\mathsf{Diff}}

\newcommand{\MC}{\mathsf{MC}}

\newcommand{\NN}{\ensuremath{\mathbb N}}
\newcommand{\ZZ}{\ensuremath{\mathbb Z}}
\newcommand{\QQ}{\ensuremath{\mathbb Q}}
\newcommand{\RR}{\ensuremath{\mathbb R}}
\newcommand{\TT}{\ensuremath{\mathbb T}}

\newcommand{\pd}[1]{\frac{\partial}{\partial #1}} %\pd{x}

\renewcommand{\graph}{\mathrm{graph}}

\newcommand{\lione}{$L_{\infty}[1]$-algebra }

\begin{document}

\author{Florian Sch\"atz}
\email{florian.schaetz@gmail.com}
\address{Eckenheimer Landstr. 136, 60318 Frankfurt am Main, Germany.
}
  
 \author{Marco Zambon}
\email{marco.zambon@kuleuven.be}
\address{KU Leuven, Department of Mathematics, Celestijnenlaan 200B box 2400, 3001 Leuven, Belgium.}

%\date{\today}

 \title[Gauge equivalences for foliations and   pre-symplectic structures]{Gauge equivalences\\ for foliations and   pre-symplectic structures}

\begin{abstract}
We consider the deformation theory of two kinds of geometric objects: foliations on one hand, pre-symplectic forms on the other. For each of them, we prove that the geometric notion of equivalence given by isotopies agrees with the algebraic   notion of gauge equivalence obtained from the $L_{\infty}$-algebras governing these deformation problems. 
\end{abstract}

\maketitle

\tableofcontents

\section*{\textsf{Introduction}}

Pre-symplectic forms are   closed 2-forms whose kernel has constant rank. 
They arise for instance {by} restricting symplectic forms to coisotropic submanifolds, such as the zero level sets of moment maps.  {They are precisely the 2-forms which admit local coordinates making them constant, and they can be also regarded as transversally symplectic foliations.}
The deformation theory of pre-symplectic forms -- unlike the one of symplectic forms -- is non-trivial, due to the fact that two conditions have to be preserved simultaneously: the closeness condition, and the constant rank one. In a previous publication \cite{SZPre}
we showed that pre-symplectic deformations of a pre-symplectic manifold $(M,\eta)$ are governed by an $L_{\infty}$-algebra $L_{\infty}(M,\eta)$ whose  
brackets are trivial except possibly for 
those of arity one, two and three. 

Further, in \cite{SZPre} we related the $L_\infty$-algebra $L_{\infty}(M,\eta)$ to the $L_{\infty}$-algebra governing  deformations of involutive distributions (i.e., of foliations), by means of a strict morphism. This can be seen as an enhancement of the geometric fact that to every  
pre-symplectic form there is an associated involutive distribution,  {given by} its kernel.\\

In this note we address \emph{equivalences} of deformations.
For pre-symplectic structures, there is a \emph{geometric} notion of equivalence: $\eta_1$ and $\eta_2$ are equivalent if they are related by a diffeomorphism isotopic to the identity. There is also an \emph{algebraic} notion of equivalence, namely the gauge equivalence of Maurer-Cartan elements of the $L_{\infty}$-algebra $L_{\infty}(M,\eta)$. 
 One of   {the} main results
{of this note} is Thm. \ref{thm:isotopyMC}, stating that these two notions essentially coincide when $M$ is compact.

In the same vein,   there are also a geometric and an algebraic notion of equivalence {for foliations}, which we  show to coincide  in Thm. \ref{thm:isotopyMCfol}. Even though the  {$L_\infty$-algebras controlling the deformations of foliations were} investigated already in the early 2000's, to our knowledge the equivalences are addressed here for the first time.

{We use these results to draw conclusions about infinitesimal and local properties of the moduli space of 
pre-symplectic structures (respectively foliations) modulo isotopies. 
Specifically, we address its smoothness, the dimension of the formal tangent spaces, and whether first order deformations extend to  (formal) paths of deformations. We do so in \S \ref{subsection:mod}.}

 The relation{ship} between {the} equivalences  {attached to the deformations of pre-symplectic forms and foliations, respectively}, is discussed in Remark \ref{rem:isopreG}.

We finish with a technical note.
For both pre-symplectic forms and foliations, the algebraic notion of equivalence
is somewhat hard to handle since it is expressed in terms of solutions of a PDE.
We bypass this problem by rephrasing algebraic equivalences on $M$ in terms of the product manifold $M\times \RR$, see Lemma \ref{lem:productfol} and Lemma \ref{rem:MCprod}.

\bigskip

\paragraph{\bf Acknowledgements:}

M.Z. acknowledges partial support by IAP Dygest, 
the long term structural funding -- Methusalem grant of the Flemish Government, 
the FWO under EOS project G0H4518N, the FWO research project G083118N (Belgium).  
{We thank the referee for  comments that allowed to improve the text, and Stephane Geudens for helpful conversations and useful references.}
\\

 \addtocontents{toc}{\protect\mbox{}\protect}%%%%%%%%%%Skips one line in TOC
\section{\textsf{Review: Deformations of foliations and of  pre-symplectic structures}}\label{section: pre-symplectic structures}

We recall results   {on} the deformation theory of two kinds of (interrelated) geometric objects: foliations  on one side,  pre-symplectic forms on the other.

\subsection{Deformations}
\label{subsec:defs}

 Deformations {of a given structure are often} controlled by an algebraic structure {known as} a differential graded Lie algebra or, more generally, an $L_{\infty}$-algebra.  
%We will work with the following notion, which is equivalent to the one of $L_{\infty}$-algebra (by a degree shift).
{
An $L_\infty$-algebra is a rigorous formalization of the naive notion of a Lie algebra up to homotopy. It consists of the datum of a $\ZZ$-graded vector space $V$ endowed with degree $(2-k)$ multilinear brackets $\wedge^kV \to V$ satisfying a set of quadratic relations generalizing the Jacobi identity, see  \cite{LadaStasheff} for a definition. By shifting degrees by $1$, one gets the equivalent notion of an $L_\infty[1]$-algebra. 
}

\begin{definition}  An \emph{$L_\infty[1]$-algebra} is a $\ZZ$-graded vector space $W$, 
equipped with a collection of graded symmetric brackets $(\lambda_k\colon W^{\otimes k} \longrightarrow W)_{k\ge1}$ of degree $1$ which satisfy 
%a collection of quadratic relations \cite{LadaStasheff}, called higher Jacobi identities.
{the (shifted) higher Jacobi identities of $L_\infty$-algebras.}
\end{definition}

%For the $L_\infty[1]$-algebras appearing in this note, all the  multibrackets vanish except possibly for the first three, namely 
%$\lambda_1,\lambda_2,\lambda_3$. We provide the definition of Maurer-Cartan element only for such 
%$L_\infty[1]$-algebras.
{We provide the definition of Maurer-Cartan element only in the case of  $L_\infty[1]$-algebras for which all but finitely many  multibrackets vanish.} For the $L_\infty[1]$-algebras appearing in this note, all the  multibrackets vanish except possibly for the first three, namely $\lambda_1,\lambda_2,\lambda_3$.

\begin{definition}
A \emph{Maurer-Cartan element} of an $L_\infty[1]$-algebra {$(W,\lambda_1,\dots,\lambda_n)$, where $n\in \NN$,} is a degree zero element $\beta\in W$ such that
$$\lambda_1(\beta)+\frac{1}{2!} \lambda_2(\beta,\beta) 
{+ \cdots +\frac{1}{n!} \lambda_n(\beta,\cdots,\beta)=0}.$$
\end{definition}

Given a (algebraic or geometric) object,  its deformations are often governed by an $L_\infty[1]$-algebra $W$, in the following strong sense: the Maurer-Cartan elements of $W$ are in natural  bijection with the ``small'' deformations of that object. 
For instance, given an integrable distribution $K$ on a manifold $M$, fix an auxiliary distribution $G$ with $K\oplus G=TM$.
In Prop. \ref{prop:Ji}
 we   display an $L_\infty[1]$-algebra whose Maurer-Cartan elements are naturally identified with integrable distributions  which are transverse to $G$.
We  do the same for pre-symplectic forms in Thm. \ref{theorem: main result}.
{To the best of our knowledge, the first geometric instance of the above phenomenon in which all the multibrackets of the $L_\infty[1]$-algebra can be non-zero, is the work of Oh-Park \cite{OP} about deformations of coisotropic submanifolds in symplectic manifolds.
%. There it is shown that deformations of coisotropic submanifolds in symplectic manifolds which are sections of a certain tubular neighborhood are  governed by an $L_\infty[1]$-algebra for which the multibrackets of all orders are typically non-trivial.
For a general discussion of the role of $L_\infty[1]$-algebra in deformation theory, see \cite{FukDef}.}

\subsection{{Foliations and their deformations}}
\label{subsection: fols}

\subsubsection{\underline{Foliations}}
\label{subsubsection:foliations}
  \hfill \break   \vspace{-3mm}%skips a line 

 Let $M$ be a smooth manifold.
A \emph{rank $l$ foliation} is a decomposition of  $M$ into immersed submanifolds of dimension $l$, subject to a local triviality condition (see for instance \cite{LeeIntroSmooth}).

 {
The celebrated Frobenius theorem states that there is a canonical bijection between foliations on $M$ and involutive distributions on $M$, i.e. subbundles of $TM$ whose sections are stable w.r.t. the Lie bracket of vector fields. More precisely,  given a foliation $M$, the associated involutive distribution $D$ is given at any point $p$ as follows: $D_p$ is the tangent space to the immersed submanifold (leaf) of the foliation through the point $p$. Because of this bijection, in the following we will work with  involutive distributions. 
}

\subsubsection{\underline{Deformation theory of foliations}}
\label{subsection: deformations of foliations}
\hfill \break   \vspace{-3mm}%skips a line
 
We review the $L_{\infty}[1]$-algebra governing deformations of foliations.  
{The following proposition \cite[Prop. 4.6]{SZPre} summarizes results by Huebschmann \cite{Huebsch}, Vitagliano \cite{VitaglianoFol} and Xiang Ji \cite{Ji}:}
\begin{proposition}\label{prop:Ji}
Let $K$ be an involutive distribution on a manifold $M$, and let $G$ be a complement.
There is an $L_{\infty}[1]$-algebra structure on
 $\Gamma(\wedge K^*\otimes G)[1]$, whose only possibly non-vanishing brackets  {we denote by} $l_1,-l_2,l_3$, with the property that the graph of $\phi\in \Gamma(K^*\otimes G)[1]$ is involutive if{f}  $\phi$ is a Maurer-Cartan element.
\end{proposition}
{Rephrasing this proposition, taking the graph of elements of $\Gamma(K^*\otimes G)$ gives a bijection
\begin{equation*}
%\label{intro: parametrization}
\boxed{  \MC(K) \to  \{\text{involutive distributions  on $M$ transverse to $G$}\}},
\end{equation*}
where $\MC(K)$ denotes the set of Maurer-Cartan elements of $(\Gamma(\wedge K^*\otimes G)[1],l_1,-l_2,l_3)$.}
  
 The formulae for $l_1,l_2,l_3$ are as follows.   We remark that $l_1$ is the differential associated to the flat $K$-connection 
 on $G$ which, under the identification $G\cong TM/K$, corresponds to the Bott connection. (Thus the underlying cochain complex is  the one used by Heitsch \cite{Heitsch} to describe infinitesimal deformations {of foliations}.)
For all $\xi \in \Gamma(\wedge^k K^*\otimes G)[1], \psi \in \Gamma(\wedge^l K^*\otimes G)[1],
 \phi \in \Gamma(\wedge^m K^*\otimes G)[1]$ we have:

\begin{eqnarray*}
 l_1(\xi)(X_1,\dots,X_{k+1})&=&\sum_{i=1}^{k+1}(-1)^{i+1} \mathrm{pr}_G\Big[X_i\,,\,\xi(X_1,\dots,\widehat{X_i},\dots,X_{k+1})\Big]\\
 &+& \sum_{i<j}(-1)^{i+j}\xi\Big([X_i,X_j],X_1,\dots,\widehat{X_i},\dots,\widehat{X_j},\dots,X_{k+1}\Big)\\ 
  l_2(\xi,\psi)(X_1,\dots,X_{k+l})&=&(-1)^{k}\sum_{\tau\in S_{k,j}}(-1)^{\tau} \mathrm{pr}_G\Big[\xi(
 X_{\tau(1)},\dots,X_{\tau(k)})\,,\,\psi(X_{\tau(k+1)},\dots,X_{\tau(k+l)})\Big]
 \\
  &{+}&(-1)^{k(l+1)} \sum_{\tau\in S_{l,1,k-1}}(-1)^{\tau} \xi\Big(\mathrm{pr}_K\Big[\psi( X_{\tau(1)},\dots,X_{\tau(l)})\,,\,X_{\tau(l+1)}\Big],\\
  && \hspace{6.5cm} ,X_{\tau(l+2)},\dots,  X_{\tau(l+k)}\Big)
 \\
  &{-}& (-1)^{k} (\xi \leftrightarrow \psi, k\leftrightarrow l)\end{eqnarray*}
  \begin{align*}
 l_3(\xi,\psi,\phi)(X_1,\dots,&X_{k+l+m-1})=  (-1)^{m+k(l+m)}\cdot\\
\sum_{\tau\in S_{l,m,k-1}}(-1)^{\tau}  \xi\Big(\mathrm{pr}_K &\Big[
 \psi(X_{\tau(1)},\dots,X_{\tau(l)})\,,\,\phi(X_{\tau(l+1)},\dots,X_{\tau(l+m)})\Big],
,\dots,X_{\tau(l+m+k-1)}\Big) \pm\circlearrowleft
\end{align*}
Here $X_i\in \Gamma(K)$,  $\mathrm{pr}_G$ is the projection $TM=G\oplus K\to G$, and similarly for $\mathrm{pr}_K$. {We use $S_{i,j,k}$ to denote} the set of permutations $\tau$ of $i+j+k$ elements such that the order is preserved within each block:
$\tau(1)<\cdots<\tau(i), \tau(i+1)<\cdots<\tau(i+j), \tau(i+j+1)<\cdots<\tau(i+j+k)
$. The symbol
$(\xi \leftrightarrow \psi, k\leftrightarrow l)$ denotes the sum just above it, switching $\xi$ with $\psi$ and $k$ with $l$. The symbol $\circlearrowleft$ denotes cyclic permutations in $\xi,\psi,\phi$.
%\end{remark}

\subsection{Pre-symplectic structures and their deformations}

After introducing pre-symplectic structures, we review the $L_{\infty}[1]$-algebra governing their deformations, following \cite[\S 1-\S3]{SZPre}.

\subsubsection{\underline{  Pre-symplectic structures}}
\label{subsection: pre-symplectic structures}
\hfill \break   \vspace{-3mm}%skips a line  

\begin{definition}
%\label{definition: pre-symplectic}
{A \emph{pre-symplectic structure} on a manifold $M$ is an element $\eta\in  \Omega^2(M)$ such that
the rank of  $\eta^\sharp: TM \to T^*M, v\mapsto \iota_v\eta$ is constant, and so that $d\eta=0$.}

{We define  $\Presym^k(M):=\{\text{pre-symplectic structures of rank $k$ on $M$}\}$.}
\end{definition}

\begin{remark}
{
Let $(M,\eta)$ be a pre-symplectic manifold. Canonically associated to it there is a 
(constant rank) distribution
$$K:=\ker(\eta^\sharp),$$
which is involutive since $d\eta=0$, and thus tangent to a foliation on $M$.
We denote the corresponding foliated
de~Rham complex by
$\Omega(K):=(\Gamma(\wedge K^*),d_\Fol).$ 
It fits in a short exact sequence of cochain complexes
\begin{equation}
 \label{eq:ses}
\xymatrix{
0 \ar[r] &\Omega_\hor(M) \ar[r] & \Omega(M) \ar[r]^r & \Omega(K)\ar[r] & 0
},
\end{equation}
where the map $r$ is given by restriction. Notice that its kernel
 $\Omega_\hor(M)$ is generated  as a multiplicative ideal
of $\Omega(M)$ by $\Gamma(K^\circ)$, where $K^\circ \subset T^*M$ is the  annihilator of $K$.
We denote by $H_\hor(M)$ the cohomology of $\Omega_\hor(M)$.}
\end{remark}

\begin{remark}\label{rem:tpres}
{The formal tangent space to $\Presym^k(M)$ at the point $\eta$ is given by
$$ T_\eta \left(\Presym^k(M)\right)  \cong \{\alpha \in \Omega^2(M) \textrm{ closed}, r(\alpha)=0\} = \{\text{2-cocycles in } \Omega_\hor(M)).$$}
\end{remark}

\subsubsection{\underline{A parametrization inspired by Dirac geometry}}
\label{subsection: Dirac parametrization}
\hfill \break   \vspace{-3mm}%skips a line

{Fix  a finite-dimensional  vector space $V$. Any $Z\in \wedge^2 V$ can be alternatively described by means} of the linear map  $$Z^\sharp\colon V^* \to V, \quad \xi \mapsto \iota_\xi Z = Z(\xi,\cdot),$$
{where the antisymmetry of $Z$ is encoded in the following equation for $Z^\sharp$:
 $(Z^\sharp)^*=-Z^\sharp$.}
 %where $(Z^\sharp)^*$ is the dual map to $Z^\sharp$.
%$\langle Z^\sharp \xi_1,\x_2 \rangle=-\langle \xi_1,Z^\sharp  \x_2 \rangle$.
%\begin{definition}
{We define $\mathcal{I}_Z$ to consist of all $\beta\in \wedge^2 V^*$
such that $\mathrm{id} + Z^\sharp \beta^\sharp\colon V\to V$ is invertible. 
%\end{definition}
The formula
\begin{equation}\label{eq:alphabeta}
(F (\beta))^\sharp := \beta^\sharp(\id + Z^{\sharp}\beta^\sharp)^{-1}=(\id + \beta^\sharp Z^{\sharp})^{-1}\beta^\sharp
\end{equation}
determines a (non-linear) map  $F \colon \mathcal{I}_Z\to \wedge^2 V^*$, which is a diffeomorphism onto its image $\mathcal{I}_{-Z}$. Denoting} by $G$ the image of $Z^{\sharp}$, for any $\beta \in  \wedge^2V^*$  we have:  $\beta\in  \mathcal{I}_Z$ if{f} $\mathrm{id_G} + Z^\sharp \sigma^\sharp\colon G\to G$ is invertible, where $\sigma:=\beta|_{\wedge^2G}$.
%A Dirac-geometric interpretation the map $F$ is given in   \cite{SZDirac}.
 
{Now fix $\eta \in \wedge^2 V^*$ of rank $k$.
Fix also a subspace $G$ such that $G\oplus K=V$, where $K=\ker(\eta^\sharp)$.
Since the restriction $\eta\vert_G\in \wedge^2G^*$ is non-degenerate, there exists a unique
$Z \in \wedge^2 G$ such that
$ Z^\sharp=-(\eta\vert_G^\sharp)^{-1}$.}  

\begin{definition}\label{definition: Dirac exponential map}
The Dirac exponential map $\exp_{\eta}$ of $\eta$ (and for fixed $G$) is the mapping
$${\exp_{\eta}}\colon \mathcal{I}_Z \to \wedge^2 V^*, \quad \beta \mapsto \eta + 
F(\beta).$$ 
\end{definition}

Let $ r: \wedge^2 V^* \to \wedge^2 K^*$ be the restriction map; we have the natural identification $\ker(r) = \wedge^2 G^* \oplus (G^*\otimes K^*)$.
The following theorem is \cite[Thm. 2.6]{SZPre}. Item (iii) {below} asserts that, upon restriction to
$\ker(r)$, the map 
 $\exp_{\eta}$ is a submanifold chart for $(\wedge^2 V^*)_k$,   the space of  {skew-symmetric bilinear forms} on $V$ of rank $k$.

\begin{theorem}\label{theorem: almost Dirac structures}
\hspace{0cm}
\begin{enumerate}
\item[(i)] Let $\beta \in \mathcal{I}_Z$. Then 
 $\exp_{\eta}(\beta)$ lies in $(\wedge^2 V^*)_k$ if, and only if, $\beta$ lies in $\ker(r)=(K^*\otimes G^*)\oplus \wedge^2G^*$.
\item[(ii)] 
Let $\beta = (\mu,\sigma) \in \mathcal{I}_Z\cap ((K^*\otimes G^*)\oplus \wedge^2G^*)$. Then  $\exp_{\eta}(\beta)$ is the unique skew-symmetric bilinear form on $V$ with the following properties:
\begin{itemize}
\item  its restriction to $G$ equals {$(\eta+F(\sigma))|_{\wedge^2G}$}
\item its kernel is the graph of the map $Z^\sharp \mu^\sharp= -(\eta\vert_G^{\sharp})^{-1}\mu^\sharp: K\to G$.
\end{itemize}
\item[(iii)]
The Dirac exponential map $\exp_{\eta}: \mathcal{I}_Z \to \wedge^2 V^*$ restricts to a 
diffeomorphism 
$$\mathcal{I}_Z \cap (K^*\otimes G^*)\oplus \wedge^2G^* \stackrel{\cong}{\longrightarrow}
{\{\eta' \in (\wedge^2 V^*)_k|\; \ker(\eta') \text{ is transverse to } G\}.}$$
\end{enumerate}
\end{theorem}

  \subsubsection{\underline{ An $L_\infty$-algebra associated to a bivector field}}
 \label{subsection: Koszul for bivectors}
\hfill \break   \vspace{-3mm}%skips a line

{Let $Z$ be a bivector field  on a manifold $M$.}
 The following results combines \cite[Prop. 3.5]{SZPre} and \cite[Cor. 3.9]{SZPre} {(upon the improvement of the latter obtained in the proof of \cite[Cor. 1.9]{SZDirac}}).

\begin{theorem}\label{thm:Z}
There is an $L_\infty[1]$-algebra structure on $\Omega(M)[2]$, whose only possibly non-vanishing multibrackets are $\lambda_1,\lambda_2,\lambda_3$, with the property that 
for all $2$-forms $\beta \in \Gamma(\mathcal{I}_Z)$ the following statements are equivalent:
\begin{enumerate}
\item $\beta$ is a Maurer-Cartan element of $(\Omega(M)[2],\lambda_1,\lambda_2,\lambda_3)$, 
\item the $2$-form $F(\beta)$ is closed.
\end{enumerate}
\end{theorem}

\begin{remark}
The differential $\lambda_1$ is just the de Rham differential. The binary bracket $\lambda_2$ is given (up to signs) by the Koszul bracket associated to $Z$. The trinary bracket $\lambda_3$ is obtained by contracting with $\frac{1}{2}[Z,Z]$, in particular it vanishes when $Z$ is a Poisson bivector field. We refer the interested reader to \cite{SZPre} for the details.
In this note we do not need the explicit formulae for these multibrackets.
\end{remark}

\subsubsection{\underline{The Koszul $L_\infty$-algebra of a pre-symplectic manifold}}
 \label{subsection: Koszul for pre-symplectic}
\hfill \break   \vspace{-3mm}%skips a line

{Now let $(M,\eta)$ be a pre-symplectic manifold.
We fix a subbundle $G$ such that $G\oplus K=TM$, where $K$ is the kernel
of $\eta$. Denote by $Z$ the bivector field on $M$ given by}
$ Z^\sharp = -(\eta\vert_G^\sharp)^{-1}$. The following is \cite[Thm. 3.17]{SZPre}.

\begin{theorem}\label{theorem: construction - Koszul L-infty}
The $L_{\infty}[1]$-algebra structure  
 on $\Omega(M)[2]$ associated to the bivector field $Z$, see Thm. \ref{thm:Z}, maps $\Omega_\hor(M)[2]$ to itself.
The subcomplex $\Omega_\hor(M)[2] \subset \Omega(M)[2]$ therefore inherits the structure of an
$L_\infty[1]$-algebra, which we call the {\em Koszul $L_\infty[1]$-algebra} of $(M,\eta)$. 
\end{theorem}

From the above one obtains the following main result \cite[Thm. 3.19]{SZPre}:

\begin{theorem}\label{theorem: main result}
Let $(M,\eta)$ be a pre-symplectic manifold. The choice of a complement $G$ to the kernel of $\eta$ determines a bivector field $Z$ by requiring $Z^\sharp=-(\eta^\sharp \vert_G)^{-1}$. 
Suppose $\beta$ is a $2$-form on $M$, which lies in $\mathcal{I}_Z$.
The following statements are equivalent:

\begin{enumerate}
\item $\beta$ is a Maurer-Cartan element of the Koszul \lione $\Omega_\hor(M)[2]$ of $(M,\eta)$,  
\item The image of $\beta$ under the map $\exp_\eta$ is a pre-symplectic structure
of the same rank as $\eta$.
\end{enumerate}
\end{theorem}

{Here 
$$ \exp_\eta: \mathcal{I}_Z \cap ((K^*\otimes G^*)\oplus \wedge^2 G^*) \to (\wedge^2 T^*M)_k$$
is obtained assembling (for every tangent space of $M$) the maps introduced in Def. \ref{definition: Dirac exponential map}, followed by restriction.
By the above theorem, together with Thm. \ref{theorem: almost Dirac structures}, this map induces a {bijection}}
\begin{equation}\label{intro: parametrization}
\boxed{ \exp_\eta: \Gamma(\mathcal{I}_Z)\cap \MC(\eta) \to  
 \{\eta' \in \Presym^k(M)|\; \ker(\eta') \text{ is transverse to } G\}
}
\end{equation}
{where $\MC(\eta)$ denotes the Maurer-Cartan set of
the Koszul $L_{\infty}[1]$-algebra.}

\subsection{A strict morphism relating deformations}
\label{subsection: strict morphism}

For every pre-symplectic form there is an associated involutive distribution, namely its kernel, and hence by the Frobenius theorem there is also an associated foliation. Following \cite[\S4]{SZPre}, we now show that this can be viewed as the map on Maurer-Cartan elements induced by a  strict morphism of $L_\infty[1]$-algebras.

{Let $(M,\eta)$ be a pre-symplectic manifold, choose
a subbundle $G$ such that $G\oplus K=TM$ where $K=\ker(\eta^\sharp)$, and denote by $Z$ the bivector field  given by $ Z^\sharp = -(\eta\vert_G^\sharp)^{-1}$.}
The following is \cite[Prop. 4.6]{SZPre}, and states that there is a strict morphism from the $L_\infty[1]$-algebra governing the deformations of the pre-symplectic form $\eta$ to the $L_\infty[1]$-algebra governing the deformations of the involutive distribution $K$.  Here $F^{2}(\Omega(M)):=\Omega_\hor(M)\cdot \Omega_\hor(M)$ is an $L_{\infty}[1]$-ideal of $\Omega_\hor(M)[2]$, {we denote $\Omega(K,G):=\Gamma(\wedge K^*\otimes G)$, and similarly for $\Omega(K,G^*)$.}
\begin{theorem}\label{theorem: characteristic foliation}
 The composition 
$$\Omega_\hor(M)[2] \rightarrow \faktor{\Omega_\hor(M)[2]}{F^2(\Omega(M))[2]} \cong \Omega(K,G^*)[1] \stackrel{Z^{\sharp}[1]}{\longrightarrow} \Omega(K,G)[1]$$
is a strict morphism of $L_\infty[1]$-algebras,
where the domain is the Koszul $L_{\infty}[1]$-algebra with multibrackets 
$\lambda_1,\lambda_2,\lambda_3$, see Theorem \ref{theorem: construction - Koszul L-infty}, and
 the target $\Omega(K,G)[1]$ is endowed with the multibrackets
$l_1,-l_2,l_3$, see Prop. \ref{prop:Ji}.
\end{theorem}

 \addtocontents{toc}{\protect\mbox{}\protect}
\section{\textsf{Equivalences}}\label{section: equiv}

All the original results of this note are contained in this section.  In  \S\ref{subsection:equideffol}
 we  show that the gauge equivalence of foliations essentially agrees with the geometric notion of equivalence given by isotopies, at least in the compact case 
(see Thm. \ref{thm:isotopyMCfol}).
In the long \S\ref{subsec:obspre} 
we obtain an analog result for pre-symplectic forms 
(see Thm. \ref{thm:isotopyMC}).

\subsection{Gauge equivalence}
\label{subsec:algequiv}

On the set of Maurer-Cartan elements of an $L_\infty[1]$-algebra there is a canonical equivalence relation. We recall it now for $L_\infty[1]$-algebras $W$  for which only the first three multibrackets
$\lambda_1,\lambda_2,\lambda_3$ might be non-zero  {and, as in the two cases of interest for this note, which   satisfy the following conditions: $W$ consists of smooth sections of a vector bundle, and each multibracket is a multidifferential operator.}

 \begin{definition}\label{def:eqMC}
 Two Maurer-Cartan elements $\beta_0$ and $\beta_1$ 
 are \emph{gauge-equivalent} if there is 
 \begin{itemize}
 \item a  smooth  one-parameter family $(\beta_t)_{t\in [0,1]}$ of {Maurer-Cartan elements}  of $W$, agreeing with the given ones at $t=0$   and $t=1$, 
\item a smooth one-parameter family $(\alpha_t)_{t\in [0,1]}$ of degree $-1$ elements of $W$,
\end{itemize}  
  such that
\begin{equation}\label{eq:equivMC}
\frac{\partial}{\partial t} \beta_{t}=\lambda_1(\alpha_{t})+ \lambda_2(\alpha_{t},\beta_{t} )+\frac{1}{2}\lambda_3(\alpha_t,\beta_{t},\beta_{t}).
\end{equation}
\end{definition}
As indicated by the name, gauge equivalences define an equivalence relation on the set of Maurer-Cartan elements of $(W,\lambda_1,\lambda_2,\lambda_3)$.
{This equivalence relation is also called ``homotopy equivalence'' in the literature. It can be also phrased as follows:
two Maurer-Cartan elements  of $W$ 
 are equivalent if{f} they are of the form $B|_{(t=0, dt=0)}$ and $B|_{(t=1, dt=0)}$
 for some  Maurer-Cartan element $B$ of the $L_{\infty}[1]$-algebra
 $W\otimes \Omega^{\bullet}([0,1])$, where $\Omega^{\bullet}([0,1])$ denotes the cochain complex of differential forms on the unit interval.
}
{The differential on $W\otimes \Omega^\bullet([0,1])$ is given by the sum of $\lambda_1$ and the de Rham differential, while $\lambda_2$ and $\lambda_3$ are extended to $W\otimes \Omega^\bullet([0,1])$ in a $\Omega^\bullet([0,1])$-linear fashion with the help of the wedge product on differential forms.}

\subsection{Equivalences of  foliations}
 \label{subsection:equideffol}
Given an involutive  distribution on $M$,  choose a complement $G$.
We saw in Prop. \ref{prop:Ji} that there is an $L_{\infty}[1]$-algebra   
whose Maurer-Cartan elements correspond bijectively 
to involutive distributions transverse to $G$. In this subsection 
we relate the induced gauge equivalence with the geometric notion of equivalence given by isotopies, see Thm. \ref{thm:isotopyMCfol} and the text following it.

\subsubsection{\underline{Isotopic foliations}}
\label{subsubsec:isotopicfol} 
 \hfill \break   \vspace{-3mm}%skips a line

We  introduce the following geometric  notion of equivalence for involutive distributions on a manifold $M$.

\begin{definition}\label{definition: isotopy of fol}
The group   $\Diff_0(M)$ of {diffeomorphisms isotopic to the identity} acts on 
the set of involutive distributions on $M$ from the left, as follows: $f$ maps   $D$ to $f_*(D)$.

We call two involutive distributions $D$ and $\tilde{D}$ \emph{isotopic} if they lie in the same orbit of this action.
\end{definition}

\subsubsection{\underline{Gauge-equivalences  of foliations}}\label{subsubsection:gaugefol}
\hfill \break   \vspace{-3mm}%skips a line
 
We fix  an involutive distribution $K$ on a manifold $M$, and a complement $G$.
In Prop. \ref{prop:Ji} we presented  an $L_{\infty}[1]$-algebra   $(\Gamma(\wedge K^*\otimes G)[1],l_1,-l_2,l_3)$ governing the deformations of the involutive distribution $K$. 

By Def. \ref{def:eqMC}, two  Maurer-Cartan elements $\Phi_0, \Phi_1$ are gauge equivalent if there is 
 a smooth    family $(\Phi_t)_{t\in [0,1]} \subset \Gamma(K^*\otimes G)$ of {Maurer-Cartan elements} interpolating between them and a smooth  family $(Y_t)_{t\in [0,1]}$ in $\Gamma(G)$
such that
$$
\frac{d}{dt}{\Phi_t}=l_1(Y_t)-l_2(Y_t,{\Phi_t})+\frac{1}{2}l_3(Y_t,{\Phi_t},{\Phi_t}).
$$
Using the formulae given after Prop. \ref{prop:Ji}, one sees that the above equation reads as follows, for all sections ${X}$ of $K$:
\begin{equation}\label{eq:ddtphi}
\frac{d}{dt}{\Phi_t}({X})=-\mathrm{pr}_G[Y_t,{X}+{\Phi_t}({X})]+{\Phi_t}(\mathrm{pr}_K[Y_t,{X}+{\Phi_t}({X})]).
\end{equation}
  We give an equivalent characterization of this equation in terms of the product manifold $M\times [0,1]$.

\begin{lemma}\label{lem:productfol}
Let $(\Phi_t)_{t\in [0,1]}$ be a smooth one-parameter family in $\Gamma(K^*\otimes G)$ and $(Y_t)_{t\in [0,1]}$ a  smooth  one-parameter family in $\Gamma(G)$. Denote by $D$   the distribution on $M\times [0,1]$ given by 
 \begin{equation*}
D_{(p,t)}:=\{v+\Phi_t(v): v\in K_p\}\oplus \mathrm{Span}(\frac{\partial}{\partial t}+Y_t|_{p}).
\end{equation*}
 Then $D$ is involutive if{f}
$\Phi_t$ consists of Maurer-Cartan elements and 
  the differential equation  \eqref{eq:ddtphi} is satisfied.
\end{lemma}

 \begin{remark}\label{rem:graphkr}
 The distribution $D$ is the graph of the vector bundle  map 
$K\oplus \RR\to G$ (where both are viewed as subbundles of $T(M\times [0,1])$) which, at points of $M\times \{t\}$, is the sum of $\Phi_t$ and of the map $1 \mapsto Y_t$.
\end{remark} 

\begin{proof}
Taking sections of $D$ lying in its first or second summand, and using the projection $M\times [0,1] \to [0,1]$, one sees that
the distribution $D$ is involutive if{f} the following two conditions are satisfied: $\mathrm{graph}(\Phi_t)$ is involutive for all $t$,  and for all $X\in \Gamma(K)$   the vector field
\begin{equation}\label{eq:bracketMR}
  [X+\Phi_t(X), \frac{\partial}{\partial t}+Y_t]
\end{equation} on $M\times [0,1]$ lies in $D$. 
Since the vector field   \eqref{eq:bracketMR} projects to zero  under $M\times [0,1] \to [0,1]$, it has no $ \frac{\partial}{\partial t}$-component,   hence the second condition is equivalent to   \eqref{eq:bracketMR} lying in $\mathrm{graph}(\Phi_t)$. This means exactly that the differential equation  \eqref{eq:ddtphi} is satisfied, as one sees writing out \eqref{eq:bracketMR} as $[X+\Phi_t(X), Y_t]- \frac{\partial}{\partial t}\Phi_t(X)$ and noticing that the last term lies in $G$. 
\end{proof}

\begin{proposition}\label{prop:explicitpush}
Let $(\Phi_t)_{t\in [0,1]}$ be any smooth one-parameter family of Maurer-Cartan elements of $\Gamma(K^*\otimes G)$ and $(Y_t)_{t\in [0,1]}$ a    one-parameter family in $\Gamma(G)$, such that  equation  \eqref{eq:ddtphi} is satisfied. Denote by $(\varphi_t)_{t\in [0,1]}$ the flow of $(Y_t)_{t\in [0,1]}$.
We have
$$\mathrm{graph}({\Phi_t})=(\varphi_t)_*(\mathrm{graph(\Phi_0)})$$
at all times $t\in [0,1]$ for which the flow is defined.
\end{proposition}

\begin{proof}
Consider the distribution $D$ appearing in Lemma \ref{lem:productfol}.

We make the following claim: 
For all $p\in M$ and times $t$ for which   $\varphi_t$ is defined, we have
\begin{equation}\label{eq:Dpt}
D_{(p,t)}=((\varphi_t)_*\mathrm{graph}(\Phi_0))_p\oplus \mathrm{Span}(\frac{\partial}{\partial t}+Y_t|_{p}).
\end{equation}
The claim implies the conclusion, using Remark \ref{rem:graphkr}.

We now prove the claim. The distribution $D$  is involutive, by  Lemma \ref{lem:productfol}. 
Consider the vector field $\bar{Y}$ on $M\times \RR$ given by $\bar{Y}_{(p,t)}=
\frac{\partial}{\partial t}+Y_t|_p$. It is a section of $D$, hence its flow preserves $D$.
In particular, the time $t$ flow $F_t$ of $\bar{Y}$ satisfies $D|_{M\times \{t\}}=(F_t)_*(D|_{M\times \{0\}})$. The r.h.s. agrees with the r.h.s. of eq. \eqref{eq:Dpt}, as can be computed easily using the fact that 
the flow $F_t$ reads $$(p,s)\mapsto (\varphi_{t+s}((\varphi_s)^{-1}(p)),s+t).$$
\end{proof}

 We can now prove the main statement of this subsection:
 
\begin{theorem}\label{thm:isotopyMCfol}
Let $M$ be a {compact}   manifold with an involutive distribution $K$, and fix a complement $G$. 
%Suppose $M$ is compact. 
Fix Maurer-Cartan elements $\Phi_0, \Phi_1$ of  $(\Gamma(\wedge K^*\otimes G)[1],l_1,-l_2,l_3)$.
 The following is equivalent:
\begin{itemize}
\item[(1)]   $\Phi_0$ and  $\Phi_1$ are   gauge equivalent
Maurer-Cartan elements  
of $(\Gamma(\wedge K^*\otimes G)[1],l_1,-l_2,l_3)$
\item[(2)] There is a diffeomorphism $\varphi$ of $M$ 
 such that $$\varphi_*(\mathrm{graph}(\Phi_0))=\mathrm{graph}(\Phi_1),$$ and so that $\varphi$
 is connected to the identity by an isotopy  $(\varphi_t)_{t\in [0,1]}$ with the property that, for all $t\in [0,1]$, 
 $(\varphi_t)_*(\mathrm{graph}(\Phi_0))$ is transverse to $G$.
\end{itemize}
\end{theorem} 
Thm. \ref{thm:isotopyMCfol} shows that 
 two involutive distributions transverse to $G$ are gauge equivalent if{f} they lie in the same connected component of 
$$\{\text{involutive distributions transverse to $G$}\}\cap  \text{(a $\Diff_0(M)$-orbit)}.$$

\begin{proof}
``$(1)\Rightarrow (2)$'': apply Prop. \ref{prop:explicitpush}, noticing that the compactness of $M$ assures that all flows are defined on $[0,1]$.

``$(2)\Rightarrow (1)$''. 
Consider the diffeomorphism 
$\bar{\varphi}$ of $M\times [0,1]$ given by 
$$(p,t)\mapsto (\varphi_t(p),t).$$
Consider the distribution on $M\times [0,1]$ given by $\mathrm{graph}(\Phi_0)\oplus \mathrm{Span}(\frac{\partial}{\partial t})$. It is an involutive distribution, {being the product of two involutive ones}. One computes that its push-forward by $\bar{\varphi}$ is 
the distribution which at each point $(p,t)\in M\times [0,1]$ reads
$$(\varphi_t)_*(\mathrm{graph}(\Phi_0)) \oplus \mathrm{Span}(\frac{\partial}{\partial t}+V_t|_{p})$$
where\footnote{Explicitly, $V_t(p)=\frac{d}{ds}|_{s=t}(\varphi_s(\varphi_t^{-1}(p)))$.} 
$V_t:=(\frac{d}{dt}\varphi_t)$. 
This distribution is exactly the distribution $D$ 
appearing in Lemma \ref{lem:productfol} associated to the data $(\Phi_t,Y_t)_{t\in [0,1]}$,
where 
\begin{itemize}
\item[]    $\Phi_t\in \Gamma( K^*\otimes G)$ is defined by the requirement that $\mathrm{graph}(\Phi_t)=(\varphi_t)_*(\mathrm{graph}(\Phi_0))$,
\item[]   $Y_t\in \Gamma(G)$ is the image of $V_t$ under the projection $TM\to G$ with kernel $(\varphi_t)_*(\mathrm{graph}(\Phi_0))$.
\end{itemize}
Hence $D$ is involutive. Using Lemma \ref{lem:productfol} we obtain $(1)$.
\end{proof}

  \begin{remark}\label{rem:folorbit} [On isotopies tangent to $G$]
 Given a diffeomorphism $\varphi$ as in Thm. \ref{thm:isotopyMCfol} (2), one can constrict a diffeomorphism $\varphi'$ with the same properties and so that the isotopy $(\varphi'_t)_{t\in [0,1]}$  in addition satisfies $(\frac{d}{dt}\varphi'_t)\in \Gamma(G)$. This geometric fact follows from the proof of the implication ``$(1)\Rightarrow (2)$'' in the above theorem.
\end{remark}

\begin{remark}
An involutive distribution $K$ on a manifold $M$ is a Lie subalgebroid of $TM$, hence the vector bundle $K$ acquires a Lie algebroid structure with injective anchor map.
Consider the DGLA of Crainic-Moerdijk \cite{MR2443928} governing deformations of this Lie algebroid structure. Its gauge-equivalences were determined by La Pastina-Vitagliano \cite[\S1]{PastinaVitagliano} when $M$ is compact, and read as follows: two Lie algebroid structures on $K$ with injective anchor are equivalent if{f} the images of their anchor maps are isotopic. The DGLA of Crainic-Moerdijk
 is $L_{\infty}$-quasi-isomorphic to the $L_{\infty}$-algebra of Prop. \ref{prop:Ji} \cite[App. C]{VitaglianoFol}. Following a suggestion by Luca Vitagliano, one might try to show that the gauge-equivalence classes of Maurer-Cartan elements of the former DGLA  are in bijection with those of the latter $L_{\infty}$-algebra, thus giving another approach to the deformation problem of foliations modulo isotopy. To do so, a necessary step is to check that the $L_{\infty}$-quasi-isomorphism relating them is sufficiently well-behaved as not to give rise to any  convergence problems.
\end{remark}

\subsection{Equivalences of pre-symplectic structures}
\label{subsec:obspre}

Given a pre-symplectic form $\eta$ on $M$ and a choice a complement to its kernel, we saw in \S\ref{subsection: Koszul for pre-symplectic} that there is an $L_{\infty}[1]$-algebra structure on $\Omega_\hor(M)[2]$ whose Maurer-Cartan elements parametrize pre-symplectic forms nearby $\eta$, via the Dirac exponential map $\beta\mapsto \exp_{\eta}(\beta)$ of Def. \ref{definition: Dirac exponential map}.
In this subsection we relate the gauge equivalence of Maurer-Cartan elements with the geometric notion of isotopic pre-symplectic forms, see Thm. \ref{thm:isotopyMC} and the text following it.
   
\subsubsection{\underline{Isotopic pre-symplectic structures}}
\label{subsubsec:isotopicpre}

 We start introducing a notion of equivalence of pre-symplectic structures on a manifold $M$, which is natural from the geometric point of view.
 
\begin{definition}\label{definition: isotopy of pre-symplectic forms}
The group  $\Diff_0(M)$ of diffeomorphisms isotopic to the identity  acts on 
$\Presym^k(M)$ from the right via $\eta \cdot f := f^* \eta$.

We call two pre-symplectic structures $\eta$ and $\tilde{\eta}$ \emph{isotopic} if they lie in the same orbit of this action, and then write $\eta \sim \tilde{\eta}$.

Furthermore, we denote the set of orbits by $\Presym^k(M)/\Diff_0(M)$.
\end{definition}

{We will need} a  reformulation of the equivalence relation $\sim$:

\begin{proposition}\label{prop:isotopy}
Suppose $M$ is compact.
Two pre-symplectic structures $\eta$ and $\tilde{\eta}$ on $M$ are isotopic, if and only if
there is a smooth one-parameter family of pre-symplectic structures $(\eta_t)_{t\in [0,1]}$ joining them, such that
the variation $\frac{d}{dt}\eta_t$ equals $d{\gamma}_t$, with ${\gamma}_t$ a section of $(\ker(\eta_t))^\circ$ depending smoothly on $t$.
\end{proposition}

{Above, $(\ker(\eta_t))^\circ\subset T^*M$ is the annihilator of $\ker(\eta_t)\subset TM$.}

\begin{proof}
Assume that $\eta$ and $\tilde{\eta}$ are isotopic via $(f_t)_{t\in [0,1]}$, in particular $f_1^*\eta=\tilde{\eta}$. Then the smooth one-parameter family $\eta_t:=(f_t)^*\eta$ satisfies the requirements of the proposition
since 
$$\frac{d}{dt}\eta_t = (f_t)^*(\Lie_{X_t}{\eta}) = d\iota_{(f_t^{-1})_*X_t}\eta_t,$$
and $\iota_{(f_t^{-1})_*X_t}\eta_t$ lies in $(\ker(\eta_t))^\circ$. Here $X_t$ is the time-dependent vector field associated to the isotopy.

One the other hand, if we are given a families $(\eta_t)_{t\in [0,1]}$ and $({\gamma}_t)_{t\in [0,1]}$ as specified in the proposition, we can apply Moser's trick. In more detail, we make the Ansatz
\begin{eqnarray*} 
0 = \frac{d}{dt}(g_t^*\eta_t) = g_t^*(d\iota_{X_t}\eta_t + \frac{d}{d t}\eta_t) = g_t^*d(\iota_{X_t}\eta_t + {\gamma}_t) ,
\end{eqnarray*}
for the isotopy $(g_t)_{t\in [0,1]}$ generated by $X_t$.
Now we can find a one-parameter family of vector fields $X_t$ such that
$$ \iota_{X_t}\eta_t + {\gamma}_t = 0,$$
since ${\gamma}_t$ lies in the image of $\eta_t^{\sharp}$.
Observe that the kernels of $\eta_t$, {as $t$ ranges over $[0,1]$}, form a vector bundle over $M\times [0,1]$ and we can choose
a complementary subbundle to it inside the pull-back of $TM$. Requiring that $X_t$ takes values in this subbundle uniquely
determines the one-parameter family $X_t$ (this shows in particular that $X_t$ can be chosen in a smooth manner).
Since $M$ is compact, $(X_t)_{t\in [0,1]}$ will integrate to an isotopy $(g_t)_{t\in [0,1]}$. Setting $f_t:=g_t^{-1}$ yields the desired isotopy satisfying $\eta_t=f_t^*\eta$, in particular $\tilde{\eta}=f_1^*\eta$.
\end{proof}
  
\begin{remark}\label{remark: tangent space to pre-symplectic moduli}
Let us determine the formal tangent space to $\Presym^k(M)/\Diff_0(M)$
at the equivalence class of $\eta$.
By Remark \ref{rem:tpres}, the formal tangent space of $\Presym^k(M)$ at $\eta$ can be identified with the closed $2$-forms on $M$ whose restriction to $K=\ker(\eta)$ is zero.
On the other hand, by Prop. \ref{prop:isotopy}, the equivalence class of $\eta$ is infinitesimally
modelled by $d \beta$ for $\beta \in \Gamma(K^\circ)$.
As the quotient of these two vector spaces, and hence as the candidate for
$T_{[\eta]}\left(\Presym^k(M)/\Diff_0(M)\right)$, we therefore
find $H^2_\hor(M)$.
\end{remark}

\subsubsection{\underline{Gauge-equivalences of two-forms}}\label{subsubsection:gaugetwoforms}
\hfill \break   \vspace{-3mm}%skips a line

Fix a bivector field $Z$ on $M$.
Recall that associated to $Z$,  we constructed an $L_\infty[1]$-algebra
structure on $\Omega(M)[2]$ with structure maps $\lambda_1$, $\lambda_2$ and $\lambda_3$, in Thm. \ref{thm:Z}. Here we give a characterization of when two sufficiently small Maurer-Cartan elements
are gauge equivalent  in terms of the {map $F$ introduced in eq. \eqref{eq:alphabeta}, see   Prop. \ref{prop:bijequiv} below.
{This will be used to prove Thm. \ref{thm:isotopyMC}, which gives a geometric characterization of when   two sufficiently small Maurer-Cartan elements of $\Omega_\hor(M)[2]$
are gauge equivalent. In the latter statement, ``sufficiently small'' means that the two  Maurer-Cartan elements take value in the neighborhood $\mathcal{I}_Z$ of the zero section of $\wedge^2 T^*M$  introduced before eq. \eqref{eq:alphabeta},  and can be joined by a smooth curve of  Maurer-Cartan elements taking values in  $\mathcal{I}_Z$.
}

 For given Maurer-Cartan elements $\beta_0,\beta_1\in \Omega^2(M)[2]$, we defined gauge equivalence 
in Def. \ref{def:eqMC}, requiring the existence of smooth  one-parameter families $(\beta_t)_{t\in [0,1]}\subset \Omega^2(M)[2]$ of Maurer-Cartan elements   and $(\alpha_t)_{t\in [0,1]}\subset \Omega^1(M)[2]$ satisfying eq. \eqref{eq:equivMC}.
We characterize this in terms of the product manifold $M\times [0,1]$, just as we did in Lemma \ref{lem:productfol}.

\begin{lemma}\label{rem:MCprod}
Two Maurer-Cartan elements $\beta_0$ and $\beta_1$
are  gauge equivalent if{f} there is
\begin{itemize}
\item a smooth one-parameter family $(\beta_t)_{t\in [0,1]}$ in $\Omega^2(M)[2]$ joining $\beta_0$ to $\beta_1$
\item a smooth one-parameter family $(\alpha_t)_{t\in [0,1]}$ in $\Omega^1(M)[2]$
\end{itemize}
 such that\footnote{Explicitly, at a point $(p,t)\in M\times [0,1]$ the above 2-form reads
  $(\beta_t)_p+dt \wedge (\alpha_t)_p$.}
 $$ \beta_t + dt\wedge \alpha_t  \in \Omega^2(M\times [0,1])[2]$$
is a Maurer-Cartan element in the   $L_{\infty}[1]$-algebra 
$$(\Omega(M\times [0,1])[2], \lambda_1 + d_{[0,1]},\lambda_2,\lambda_3).$$ Here
the operator $d_{[0,1]}$ is the de Rham differential on the interval $[0,1]$, while
$\lambda_i$ denotes the $\Omega([0,1])$-linear extensions of $\lambda_i$ from
$\Omega(M)$ to $\Omega(M\times [0,1])$, and the Koszul sign convention is understood.  
\end{lemma}
 \begin{proof}
By a   straightforward computation,  $\beta_t + dt\wedge \alpha_t$ is a Maurer-Cartan element if{f}
 $(\beta_t)_{t\in [0,1]}$ and  $(\alpha_t)_{t\in [0,1]}$  satisfy eq. \eqref{eq:equivMC}.
\end{proof}

 Consider the manifold $M\times [0,1]$, equipped with the  bivector field
$\widetilde{Z}$ defined by $\widetilde{Z}_{(p,t)}=(Z_p,0)\in T_{(p,t)}( M\times [0,1])$.
Denote by $\widetilde{F}\colon \mathcal{I}_{\widetilde{Z}} \to \mathcal{I}_{-\widetilde{Z}}$ the bijection $\gamma \mapsto \gamma^\sharp(\id + \widetilde{Z}^{\sharp}\gamma^\sharp)^{-1}$ between neighborhoods of the origin in 
$\Omega^2(M\times [0,1])$, as in eq. \eqref{eq:alphabeta}.
The $L_\infty[1]$-algebra structure associated to $\widetilde{Z}$ by Thm. \ref{thm:Z} is exactly
$(\Omega(M\times [0,1])[2],\lambda_1 + d_{[0,1]},\lambda_2,\lambda_3)$  as
introduced in Lemma \ref{rem:MCprod}.   
Given a Maurer-Cartan element $\beta_t + dt\wedge \alpha_t \in \Omega(M\times [0,1])[2]$ 
lying in $\mathcal{I}_{\widetilde{Z}}$,
we know by Thm. \ref{thm:Z} that the image  
 \begin{equation}\label{eq:psiMxI}
 \widetilde{F}(\beta_t + dt\wedge \alpha_t) =: \widehat{\beta}_t +  dt\wedge \widehat{\alpha}_t
\end{equation}
 is closed with respect to the de Rham differential, which in turn amounts to
 \begin{equation}\label{eq:psiMxIrhs} d\widehat{\beta}_t=0 \qquad \textrm{and} \qquad \frac{\partial \widehat{\beta}_t}{\partial t} = d \widehat{\alpha}_t.
 \end{equation}
Notice that $\beta_t + dt\wedge \alpha_t \in \Omega(M\times [0,1])$ lies in $\mathcal{I}_{\widetilde{Z}}$ if{f} $\beta_t$ lies in $\mathcal{I}_Z$ for all $t$, see  \S\ref{subsection: Dirac parametrization}.  
 {Since $\widetilde{F}$} maps   $\mathcal{I}_{\widetilde{Z}}$ bijectively onto  $\mathcal{I}_{-\widetilde{Z}}$, it follows that $\widehat{\beta}_t$ lies in $\mathcal{I}_{-Z}$ for all $t$.
 
We compute explicitly the r.h.s. of eq. \eqref{eq:psiMxI}.
 
\begin{lemma}\label{proposition: algebraic map = geometric map II}
For  a family $\alpha_t$ of  $1$-forms on $M$ and   a family $\beta_t$ of  $2$-forms lying in $\mathcal{I}_Z$ we have  
$$\widetilde{F}(\beta_t + dt\wedge \alpha_t)= F(\beta_t) + dt \wedge F'(\beta_t,\alpha_t) \quad \in \Omega^2(M\times [0,1]).$$
Here for any  pair $(\beta,\alpha)\in \Gamma(\mathcal{I}_Z)\times \Omega^1(M)$ we denote
\begin{equation}\label{eq:tildes}   \widehat{\beta}^{\sharp}:=F(\beta)^\sharp = \beta^\sharp(\mathrm{id} + Z^\sharp \beta^\sharp)^{-1}  \textrm{ as in eq. \eqref{eq:alphabeta},} \qquad \widehat{\alpha}:=F'(\beta,\alpha):= (\mathrm{id} +  \beta^\sharp Z^\sharp)^{-1}\alpha.
\end{equation}
\end{lemma}

\begin{proof}
{We have}
$$\widetilde{F}(\beta_{{t}} + dt \wedge  \alpha_{{t}})^\sharp = 
[\mathrm{id} + (\beta_{{t}} + dt \wedge \alpha_t)^\sharp\widetilde{Z}^\sharp]^{-1}(\beta_{{t}} + dt \wedge \alpha_{{t}})^{\sharp}.$$
W.r.t. the splitting $T^*_{(p,t)}(M\times [0,1])=T^*_pM \oplus \RR dt$,
the isomorphism in the square bracket is lower triangular, so it can be easily inverted to obtain
$\left(\begin{array}{cc} (\mathrm{id}+\beta^{\sharp}Z^{\sharp})^{-1}  & 
0
 \\ dt\cdot\alpha^{\sharp}Z^{\sharp}   (\mathrm{id}+\beta^{\sharp}Z^{\sharp})^{-1} & \mathrm{id}_{\RR dt}     \end{array}\right).$
 Using this, it is straightforward  to verify that {the above} $2$-form on $M\times [0,1]$ simplifies to
$F(\beta_{t}) + dt \wedge  F'(\beta_t,\alpha_t)$ as stated in the proposition.
\end{proof}

In conclusion, we obtain:

\begin{proposition}\label{prop:bijequiv}
Fix a bivector field $Z$ on $M$. There is a bijection between:

 \begin{enumerate}
\item  One-parameter families $(\beta_t)_{t\in [0,1]}\subset \mathcal{I}_Z$  of Maurer-Cartan elements of $(\Omega(M)[2],\lambda_1,\lambda_2,\lambda_3)$, \\
 one-parameter families $(\alpha_t)_{t\in [0,1]}$ in $\Omega^1(M)$\\
 such that
 the   differential equation
\eqref{eq:equivMC} is satisfied,

\item  One-parameter families $(\widehat{\beta}_t)_{t\in [0,1]}\subset \mathcal{I}_{-Z}$ of closed forms in $\Omega^2(M)$   and\\   
one-parameter families 
$(\widehat{\alpha}_t)_{t\in [0,1]}$ in $\Omega^1(M)$\\
 satisfying $\frac{\partial \widehat{\beta}_t}{\partial t} = d \widehat{\alpha}_t$.
\end{enumerate}
The bijection maps  $(\beta_t,\alpha_t)_{t\in [0,1]}$ to $(\widehat{\beta}_t,\widehat{\alpha}_t)_{t\in [0,1]}$ 
 as in eq. \eqref{eq:tildes}.
\end{proposition}
\begin{proof}
Given an element of the set (1), apply Lemma \ref{rem:MCprod} (and its proof) and the map $\widetilde{F}$
to obtain an element of the set (2).
This map is a bijection since $\widetilde{F}\colon \mathcal{I}_{\widetilde{Z}} \to \mathcal{I}_{-\widetilde{Z}}$ is a bijection. {The formulae for $\widehat{\beta}_t$ and $\widehat{\alpha}_t$ were obtained in 
 Lemma \ref{proposition: algebraic map = geometric map II}.}
\end{proof}
%%%

\subsubsection{\underline{Gauge equivalences of pre-symplectic forms}}
\label{subsubsection:gaugepre}
\hfill \break   \vspace{-3mm}%skips a line

Let $(M,\eta)$ be a pre-symplectic manifold. As usual, we fix a complement $G$
of the kernel $K=\ker(\eta)$ and denote by $Z$ the  bivector field 
determined by $ Z^\sharp=-(\eta\vert_G^\sharp)^{-1}$.

The subcomplex of $(\Omega(M)[2],\lambda_1,\lambda_2,\lambda_3)$ consisting of horizontal differential forms $\Omega_\hor(M)$ forms an $L_\infty[1]$-subalgebra, which we refer to as the Koszul $L_\infty[1]$-algebra of $(M,\eta)$, see Thm. \ref{theorem: construction - Koszul L-infty}.
Further we saw that the Maurer-Cartan elements of the Koszul $L_\infty[1]$-algebra 
correspond to deformations of the pre-symplectic structure $\eta$ (the rank is required to stay fixed), see Thm. \ref{theorem: main result}. 
Now we specialize  to horizontal forms the results about gauge equivalence obtained in \S \ref{subsubsection:gaugetwoforms}.\\
 
We first need a linear algebra lemma. Recall that $\widehat{\alpha}, \widehat{\beta}$ were defined in  eq. \eqref{eq:tildes}.
\begin{lemma}\label{lem:vanishes}
Let $\beta\in \Omega^2_{\hor}(M)$ lie in $\mathcal{I}_Z$. For any $\alpha \in \Omega^1(M)$ we have: $\alpha \in \Omega^1_{\hor}(M)$ if{f}
$\widehat{\alpha}$ vanishes on the kernel of
${\eta}+\widehat{\beta}$.
\end{lemma}
\begin{proof}
The kernel of ${\eta}+\widehat{\beta}$ equals the image of the  
map
$(\mathrm{id}+ Z^\sharp \beta^{\sharp}|_K)\colon K \to TM$, by Theorem \ref{theorem: almost Dirac structures}.
For all $v\in K$ we have 
$$\langle\widehat{\alpha},(\mathrm{id} + Z^\sharp \beta^\sharp)(v)\rangle
=\langle(\mathrm{id}+\beta^\sharp Z^\sharp)^{-1}(\alpha),(\mathrm{id}+Z^\sharp\beta^\sharp)(v)\rangle
=\langle \alpha,v\rangle,$$
proving the lemma. 
\end{proof} 

We can now refine Prop. \ref{prop:bijequiv} to:

\begin{proposition}\label{prop:bijequivhor}
There is a bijection between:
 
 \begin{enumerate}
\item  One-parameter families $(\beta_t)_{t\in [0,1]}\subset \mathcal{I}_Z$  of Maurer-Cartan elements of $(\Omega_\hor(M)[2],\lambda_1,\lambda_2,\lambda_3)$, \\
 one-parameter families $(\alpha_t)_{t\in [0,1]}$ in $\Omega_\hor^1(M)$\\
 such that
 the   differential equation
\eqref{eq:equivMC} is satisfied.  

\item  One-parameter families $(\widehat{\beta}_t)_{t\in [0,1]}$ in $\Omega^2(M)$   and\\   
one-parameter families 
$(\widehat{\alpha}_t)_{t\in [0,1]}$ in $\Omega^1(M)$\\
such that 
\begin{itemize}
\item $ \eta_t:=\eta+\widehat{\beta}_t$ is pre-symplectic on $M$ of rank equal to $rank(\eta)$ and with kernel transverse to $G$, 
\item $\widehat{\alpha}_t$ vanishes on $\ker(\eta_t)$, 
\item
 $\frac{\partial \eta_t}{\partial t} = d\widehat{\alpha_t}$.
\end{itemize}
  \end{enumerate}
 \end{proposition}
\begin{proof}
Use Prop. \ref{prop:bijequiv}, together with Thm. \ref{theorem: main result} with the text following it, 
and Lemma \ref{lem:vanishes}.
\end{proof}

Rephrasing the above proposition gives the main result of this subsection.

\begin{theorem}\label{thm:isotopyMC}
Let $(M,\eta)$ be a  {compact} pre-symplectic manifold, and fix a complement $G$
of $\ker(\eta)$. 
Fix 
Maurer-Cartan elements $\beta_i$ 
of $(\Omega_\hor(M)[2],\lambda_1,\lambda_2,\lambda_3)$ lying in $\mathcal{I}_Z$, and 
denote by $\eta_i:=\eta+F({\beta}_i)$ the corresponding
pre-symplectic forms, for $i=0,1$.
The following is equivalent:
\begin{itemize}
\item[(1)] The Maurer-Cartan elements  $\beta_0$ and $\beta_1$ are   gauge equivalent
{through Maurer-Cartan elements $(\beta_t)_{t\in [0,1]}$ lying in $\mathcal{I}_Z$}.
\item[(2)] There is a diffeomorphism of $M$  pulling back  $\eta_1$ to $\eta_0$
and isotopic to the identity by
 an isotopy $(\varphi_t)_{t\in [0,1]}$  such that $(\varphi_t)_*(\ker(\eta_0))$ is transverse to $G$ for all $t$.
\end{itemize}

\end{theorem}
\begin{proof}
Apply Prop. \ref{prop:bijequivhor} and Prop. \ref{prop:isotopy}.
\end{proof}
Thm. \ref{thm:isotopyMC} shows that 
 two Maurer-Cartan elements 
lying in $\mathcal{I}_Z$ are gauge equivalent  through Maurer-Cartan elements   lying in $\mathcal{I}_Z$ if{f} they lie in the same connected component of 
$$\{\text{pre-symplectic forms with kernel transverse to $G$}\}\cap  \text{(a $\Diff_0(M)$-orbit)}.$$

\begin{remark}\label{rem:isopreG}[On isotopies tangent to $G$, and relation to Thm. \ref{thm:isotopyMCfol}]

Suppose we are given  $\beta_0,\beta_1 \in \Omega^2_{\hor}(M)$ as in the first item of  Thm. \ref{thm:isotopyMC}, that is: they are gauge-equivalent via one-parameter families $(\alpha_t,\beta_t)_{t\in [0,1]}$ with 
$\beta_t\in  \mathcal{I}_Z$.

(i) An isotopy as in the second item of Thm. \ref{thm:isotopyMC} is obtained {by} integrating $Z^{\sharp}\alpha_t$. Indeed, the isotopy constructed in Prop. \ref{prop:isotopy} is obtained {by} integrating a time-dependent vector field $X_t$ satisfying $\iota_{X_t}{\eta_t}=-\widehat{\alpha}_t$.
We claim that $X_t:=Z^{\sharp}\alpha_t$ (a vector field tangent to the distribution $G$) satisfies this equation. 
The claim follows from
$$\iota_{Z^{\sharp}\alpha_t}{\eta_t}=-\alpha_t+  \beta_t^\sharp(id + Z^\sharp \beta_t^\sharp)^{-1}Z^{\sharp}\alpha_t= -(id +  \beta_t^\sharp Z^\sharp)^{-1} \alpha_t=
-\widehat{\alpha}_t,$$
where we used $\beta_t^\sharp(id + Z^\sharp \beta_t^\sharp)^{-1}Z^{\sharp}=id-(id + \beta_t^\sharp Z^\sharp )^{-1}$ in the second equality.

This implies the following geometric fact, analogous to Lemma \ref{rem:folorbit}: given a diffeomorphism $\varphi$ as in Thm. \ref{thm:isotopyMC} (2), one can construct a diffeomorphism $\varphi'$ with the same properties and so that in addition  the isotopy $(\varphi'_t)_{t\in [0,1]}$  satisfies $(\frac{d}{dt}\varphi'_t)\in \Gamma(G)$.

(ii) 
In   Thm. \ref{theorem: characteristic foliation} we displayed  a strict $L_{\infty}[1]$-morphism from $\Omega_{\hor}(M)[2]$ to  $\Gamma(\wedge K^*\otimes G)[1]$, the 
$L_{\infty}[1]$-algebra which governs deformations of involutive distributions   transverse to $G$.
Applying this morphism to $\beta_i$ ($i=0,1$) we obtain Maurer-Cartan elements $\Phi_i:=Z^{\sharp}\beta_i|_K$ of $\Gamma(\wedge K^*\otimes G)[1]$ which are gauge-equivalent, via $(Z^{\sharp}\alpha_t, Z^{\sharp}\beta_t)_{t\in [0,1]}$. Notice that
 $\graph(\Phi_i)=\ker(\eta_i)$ by Thm. \ref{theorem: almost Dirac structures}.

On the other hand, in view of (i) above, the isotopy $(\varphi_t)_{t\in [0,1]}$ obtained integrating $Z^{\sharp}\alpha_t$ satisfies the  properties   in   Thm. \ref{thm:isotopyMC} (2). In particular, $(\varphi_1)_*(\ker(\eta_0))=\ker(\eta_1)$  and
 $(\varphi_t)_*(\ker(\eta_0))$ is transverse to $G$ for all $t\in [0,1]$. 
 
 The {two results explained in this item} are consistent with   Thm. \ref{thm:isotopyMCfol} on equivalences of foliations.
\end{remark}

\begin{remark}\label{rem:altproof}[An alternative proof of Prop. \ref{prop:bijequivhor}]
Our Prop. \ref{prop:bijequivhor} was proven specializing Prop. \ref{prop:bijequiv} to horizontal forms on $M$,   
by means of Lemma \ref{lem:vanishes}. 
 
We provide an alternative proof of Prop. \ref{prop:bijequivhor}, working on the product  
$\widetilde{M}:=M\times [0,1]$ and specializing Lemma \ref{proposition: algebraic map = geometric map II} to horizontal forms on $\widetilde{M}$. We view $\eta$ as a 2-form on $\widetilde{M}$, by taking its pullback. It is pre-symplectic, with kernel
$K\oplus \RR \partial_t$. Notice that a
 complement to this kernel in $T\widetilde{M}$ is
 (the obvious extension of) $G$.  
 
 Take one-parameter families $(\beta_t)_{t\in [0,1]}$ and $(\alpha_t)_{t\in [0,1]}$
 as in   (1) of Prop. \ref{prop:bijequivhor}.
 The 2-form $\beta_t + dt \wedge \alpha_t$ on $\widetilde{M}$ is clearly horizontal, and is a Maurer-Cartan element by Lemma  \ref{rem:MCprod}.
Hence 
by applying ${\widetilde{F}}$ as in eq. \eqref{eq:psiMxI} 
 we see (using Thm. \ref{theorem: almost Dirac structures} and Thm. \ref{theorem: main result}) that
 $$(\eta+\widehat{\beta}_t) + dt \wedge \widehat{\alpha}_t$$ is a pre-symplectic form on $\widetilde{M}$ whose kernel has dimension $\dim(K)+1$ and is transverse to $G$. 
 
 Notice that, for all $t$,  the condition $\widehat{\alpha}_t\in (\ker (\eta+\widehat{\beta}_t))^{\circ}$ holds at every point of $M$. (This follows from eq. \eqref{eq:rankkerTheta}, using that 
 $\dim(\ker(\eta+\widehat{\beta}_t))=\dim(K)$ by Thm. \ref{theorem: almost Dirac structures} (i).) Applying Cor. \ref{cor:presmxI} and using eq. \eqref{eq:kerTheta} we obtain item (2) in Prop. \ref{prop:bijequivhor}.
Clearly this reasoning can be inverted.

 \end{remark}

%%%   
\subsection{{Moduli spaces}}\label{subsection:mod}

We collect some geometric corollaries about the space of foliations (respectively pre-symplectic structures) modulo isotopy.

\subsubsection{Moduli space of foliations}

Denote by $\Foliations^k(M)$ the space of codimension $k$ involutive distributions on a manifold $M$.
Recall that $\Foliations^k(M)$ carries an action of the   group of isotopies  $\Diff_0(M)$   (Def. \ref{definition: isotopy of fol}).

Fix  an involutive distribution $K$.
 Thanks to Prop. \ref{prop:Ji} and Thm. \ref{thm:isotopyMCfol}, we know that
the formal tangent space to the moduli space $\Foliations^k(M)/\Diff_0(M)$
at the equivalence class of $K$ is
$$T_{[K]}\left(\Foliations^k(M))/\Diff_0(M)\right)\cong H^1(K,G).$$
Here  $G$ is an auxiliary distribution complementary to $K$, and $H(K,G)$ denotes the cohomology of the complex
$(\Gamma(\wedge K^*\otimes G), l_1)$ introduced just after Prop. \ref{prop:Ji}. We now show that the dimension of this tangent space can vary abruptly.

\begin{lemma}\label{lem:lambda}
Let $M=\TT^2$ be the 2-torus with ``coordinates'' $\theta_1$ and $\theta_2$. 
For any   $\lambda \in \RR$
consider the involutive distribution $K_{\lambda}:=Span\{\pd{\theta_1}+\lambda \pd{\theta_2} \}$, and its complement $G=Span\{\pd{\theta_2}\}$.
Then 
$$\dim(H^1(K_{\lambda},G))=
\begin{cases}
\infty \;\;\text{if $\lambda$ is a rational or Liouville number}\\
1 \;\;\;\;\text{if $\lambda$ is a diophantine number.}
\end{cases}
$$
  \end{lemma}

Here we call an irrational number $\lambda$ \emph{diophantine} if there exists a positive real number $s$ such that 
$\inf\{|m\lambda+n|(1+m^2)^s\}>0$, where the infimum ranges over all  $(m,n)\in \ZZ^2\setminus\{0\}$. The \emph{Liouville}
 numbers are the irrational numbers that are not diophantine; they have the property that they can be well approximated by rational numbers.

\begin{proof}
%The Bott $K$-connection on $G$ is trivial, for
By means of the line bundle isomorphism $G\cong \RR\times M$ mapping the section $\pd{\theta_2}$ to the constant section $1$,  the Bott connection on $G$ is identified  with the   foliated de Rham differential acting on functions. This means that the differential $l_1$ can be identified with the   foliated de Rham differential,  and therefore 
$H(K_{\lambda},G)$ is isomorphic to the   foliated cohomology $H(K_{\lambda})$. 

If $\lambda$ is rational, then the foliation $K_{\lambda}$ is the fiber foliation of a principal $U(1)$-bundle $M\to S^1$, and since the fibers have non-trivial first cohomology it follows that
$H^1(K_{\lambda})$ is infinite-dimensional.
If $\lambda$ is a  Liouville or diophantine number, 
for a computation of $H^1(K_{\lambda})$ using Fourier analysis we refer to
\cite[\S 2.1]{HaelfMinimal} (see also \cite[page 60]{MooreSchochet}).
\end{proof}  
  
  \begin{remark}[Obstructedness]\label{rem:obstr}
In general, the deformation problem of foliations is formally obstructed. This means that in general there is a closed element of $\Gamma(K^*\otimes G)$ -- in other words, a first order deformation --
that can not be extended to a formal curve of Maurer-Cartan elements.  
This was shown in \cite[\S 6.2]{SZPre}, by taking the underlying manifold to be the 4-torus $\TT^4$ and $K=Span\{\pd{\theta_1},\pd{\theta_2} \}$.
\end{remark}

We can now address the smoothness of the moduli space.
\begin{corollary}\label{cor:smooth}
The moduli space of codimension $k$ foliations $\Foliations^k(M))/\Diff_0(M)$  
is generally not smooth. 
\end{corollary}
\begin{proof}
We first remark that if a moduli space is smooth, then all its tangent spaces are locally isomorphic. 
Let $\lambda$ be a  diophantine number, and consider
 the involutive distribution  $K_{\lambda}$ on the   2-torus as in  Lemma \ref{lem:lambda}. If the moduli space was smooth nearby    $[K_{\lambda}]$, it would be 1-dimensional, because the formal tangent space at $[K_{\lambda}]$ is  1-dimensional by Lemma \ref{lem:lambda}. There is a sequence of rational numbers $(\lambda_n)$ converging to $\lambda$,
since the rational numbers  are dense in the real numbers. Thus we can find involutive distributions $K_{\lambda_n}$ arbitrarily close to $K_{\lambda}$   in the $C^{\infty}$-sense, and we know that the formal tangent space at each $[K_{\lambda_n}]$ is infinite-dimensional  by  Lemma \ref{lem:lambda}, thus not isomorphic to the one at $[K_{\lambda}]$. This provides a contradiction to the smoothness.
  
An alternative argument is as follows. If a moduli space is smooth nearby a given point $[K]$, then the deformations of $K$ are unobstructed. 
%\mcomment{Here I am vague, I guess  "smooth" should implies that any tangent vector can be extended a smooth curve, which moreover can be lifted to $\Foliations^k(M))$}
But we saw  that this is not always the case, in Remark \ref{rem:obstr}.
\end{proof}

\subsubsection{Moduli space of pre-symplectic structures}

We know from Remark \ref{remark: tangent space to pre-symplectic moduli} that 
the formal tangent space to $\Presym^k(M)/\Diff_0(M)$
at the equivalence class of a pre-symplectic form $\eta$ is
$$T_{[\eta]}\left(\Presym^k(M)/\Diff_0(M)\right)\cong H^2_\hor(M).$$
We show that its dimension can vary abruptly.

Recall the short exact sequence of complexes \eqref{eq:ses}. 
%We denoted the cohomology of $\Omega_\hor(M)$ by $H_\hor(M)$, 
It induces a long exact sequence in cohomology
$$\dots\to H^1(M)\to H^1(K)\overset{\delta}{\to}H^2_\hor(M)\to H^2(M)\to \dots$$
where  $H(K)$ denotes the foliated cohomology of the foliation $K$, where $H(M)$ is the de Rham cohomology of $M$, and $\delta$ is the connecting homomorphism. When  $M$ is compact\footnote{This ensures that the de Rham cohomology of $M$ is finite-dimensional.}, 
this leads to the following observation:
%we conclude that   
$H^1(K)$ is infinite-dimensional if{f} $H^2_\hor(M)$ is infinite-dimensional.
% We use this observation to prove the following lemma.

\begin{lemma}\label{lem:lambdamu}
Let $M=\TT^3$ be the 3-torus.
% with ``coordinates'' $\theta_1$ and $\theta_2$. 
For any   $\lambda, \mu \in \RR$
consider the pre-symplectic form $$\eta_{\lambda, \mu}:=d\theta_2\wedge d\theta_3-\lambda d\theta_1\wedge d\theta_3+\mu d\theta_1\wedge d\theta_2.$$

% and its complement $G=Span\{\pd{\theta_2}\}$.
Then 
$$\dim(H^2_\hor(M)) \text{ is }
\begin{cases}
\text{infinite}\;\;\text{if $(\lambda, \mu)\in \QQ^2$}\\
\text{finite} \;\;\;\;\;\text{if $(1,\lambda,\mu)$ satisfies the diophantine condition.}
\end{cases}
$$
  \end{lemma} 
 Above, following \cite[Thm. 1.3.7]{CRMFoliations} and its proof, 
%  \mcomment{
 %I am not sure if there is a typo in the definition there, since the formula is different from the one for "diophantine" in the previous subsection; 
% I should also cite \cite{LinFolTn}  by I can't find the paper online}
  we say that a vector $v\in \RR^3$ satisfies the   \emph{diophantine condition}\footnote{Notice that the expression $|  m \cdot v | (||m||)^s$, specialized to the case $v=(1,\lambda)\in \RR^2$, does not reduce exactly to the similar expression appearing just after   
Lemma \ref{lem:lambda}. However, for fixed $\lambda$, the two diophantine conditions are equivalent (see  \cite[text after eq.  (1.9)]{MoserStabMinFol} for the statement, and \cite[page 85]{MoserMinFolLNM} for a proof).} if  there exists a positive real number $s$ such that 
$\inf |  m \cdot v | (||m||)^s>0$, where the infimum ranges over all  $m \in \ZZ^3\setminus\{0\}$ and the dot denotes the dot product in $\RR^3$.

\begin{proof}
We make use of the observation just before the lemma.
The kernel of $\eta_{\lambda, \mu}$ is the involutive distribution $$K_{\lambda, \mu}:=Span\left\{\pd{\theta_1}+\lambda \pd{\theta_2} +\mu \pd{\theta_3}\right\}.$$
 
If $(\lambda, \mu)\in \QQ^2$, then $K_{\lambda, \mu}$  is the fiber distribution of a principal $U(1)$-bundle $M\to \TT^2$, so $H^1(K_{\lambda, \mu})$ is infinite-dimensional.

If $(1,\lambda,\mu)$ satisfies the    diophantine condition, then $H^1(K_{\lambda, \mu})$ is 1-dimensional. This is proven in \cite{LinFolTn} (see also \cite[Thm. 1.3.7]{CRMFoliations}),
 %taking $N=3, p=1$ there), 
 again using Fourier analysis. 
\end{proof}

 \begin{remark}[Obstructedness]\label{rem:obstrpre}
In general, the deformation problem of pre-symplectic structures is formally obstructed. 
%This means that in general there is a closed element of $\Omega^1(K,G)$ -- in other words, and first order deformation --
%that can not be extended to a formal curve of Maurer-Cartan elements.  
This was shown in \cite[\S 6.1]{SZPre}, refining the example mentioned for foliations in Remark \ref{rem:obstr}.
\end{remark}

We finish addressing the smoothness of the moduli space.
\begin{corollary}
The moduli space of rank $k$ pre-symplectic structures $\Presym^k(M)/\Diff_0(M)$  
is generally not smooth. 
\end{corollary}
\begin{proof}
We just sketch the proof, since it is analogous to the one of Corollary \ref{cor:smooth}.
One approach is to use the fact that if a moduli space is smooth, then all its tangent spaces are locally isomorphic. For this, consider the pre-symplectic form $\eta_{\lambda, \mu}$ of Lemma \ref{lem:lambdamu}  with
  $(1,\lambda, \mu)$ satisfying the diophantine condition, and notice that there is a sequence 
 $(\lambda_n, \mu_n)\in \QQ^2$ converging to $(\lambda, \mu)$.

Another approach is to use obstructed first order deformations, 
whose existence is guaranteed by Remark \ref{rem:obstrpre}.
\end{proof}

\appendix

\section{Pre-symplectic forms on $M\times [0,1]$}

We present some elementary facts about pre-symplectic forms on a product manifold, which we use in Remark \ref{rem:altproof} and which are of independent interest.
Any 2-form on $M\times [0,1]$ can be written as $$\Theta:=\eta_t+dt \wedge A_t$$ for smooth families 
$(\eta_t)_{t\in [0,1]}$ in $\Omega^2(M)$ and $(A_t)_{t\in [0,1]}$ in $\Omega^1(M)$. 
To describe its kernel in terms of $\eta_t$ and $A_t$, one has to distinguish\footnote{Indeed,  $v+c\partial_t\in \ker(\eta_t+dt \wedge A_t)$
if{f} $\iota_v A_t=0$ and $\iota_v\eta_t=-cA_t$. In the case that $A_t$ lies in $(\ker (\eta_t))^{\circ}=im(\eta_t^{\sharp})$, the second equation has solutions for all $c\in \RR$ and the first equation is automatically satisfied.}
 two cases:  
 \begin{equation}\label{eq:kerTheta} 
                     \ker(\Theta)_{(p,t)}=
\begin{cases}
\{v+c\partial_t: v\in T_pM, c\in \RR, \iota_v\eta_t=-cA_t\} \quad\text{if $A_t\in (\ker (\eta_t))^{\circ}$ at $p$}\\
\ker (\eta_t)\cap \ker (A_t) \quad\quad\quad\quad\quad\quad\quad\quad\quad\quad\quad\text{otherwise.}
                    \end{cases}
 \end{equation}
From this we see that 
 \begin{equation}\label{eq:rankkerTheta}
                    \dim(\ker(\Theta)_{(p,t)})=
\begin{cases}
\dim(\ker(\eta_t))+1 \quad\quad\quad\text{if $A_t\in (\ker (\eta_t))^{\circ}$ at $p$}\\
\dim(\ker(\eta_t))-1 \quad\quad\quad\text{otherwise.}
                   \end{cases}
 \end{equation}

It can happen that $\Theta$ has constant rank and each of the two cases applies at points of $M\times [0,1]$: take for instance $M=\RR^2$, $\eta_t=tdx\wedge dy$ and $A_t=dx$, so that $\Theta=dx\wedge(tdy-dt)$.

\begin{lemma}\label{lem:dimensions}
Consider the 2-form $\Theta=\eta_t+dt \wedge A_t$ on $M\times [0,1]$. Any two of the following conditions implies the remaining one:

(1) $\Theta$ has constant rank

(2) for all $t$, the rank of $\eta_t$ is constant and independent of $t$ 

(3) the set of points $(p,t)$ for which $(A_t)_p\in (\ker (\eta_t)_p)^{\circ}$
is either  $M\times [0,1]$ or empty.

\end{lemma}
\begin{proof}
Apply eq. \eqref{eq:rankkerTheta}.
\end{proof}

Now we bring the closeness condition into play.
Clearly a two form $\Theta=\eta_t+dt \wedge A_t$ is closed if{f} $d\eta_t=0$ and $\frac{\partial \eta_t}{\partial t}
=dA_t$ for all $t$. 
Hence from  Lemma \ref{lem:dimensions} we obtain:
\begin{corollary}\label{cor:presmxI}
Consider   smooth families
$(\eta_t)_{t\in [0,1]}$ in $\Omega^2(M)$ and $(A_t)_{t\in [0,1]}$ in $\Omega^1(M)$. 
Assume that $A_t\in (\ker (\eta_t))^{\circ}$ at every point. Then the following is equivalent:
\begin{itemize}
\item  $\eta_t+dt \wedge A_t$ is a pre-symplectic form on $M\times [0,1]$
\item for all $t$, $\eta_t$ is a pre-symplectic form on $M$ whose rank is independent of $t$, and $\frac{\partial \eta_t}{\partial t}
=dA_t$.
\end{itemize}

\end{corollary}

\bigskip
  
\bibliographystyle{habbrv} 
%\bibliography{Pres}

\end{document}